\documentclass{article}
\usepackage{amssymb,amsmath}
\usepackage{graphicx}

\def\qed{\hfill {\hbox{${\vcenter{\vbox{               
   \hrule height 0.4pt\hbox{\vrule width 0.4pt height 6pt
   \kern5pt\vrule width 0.4pt}\hrule height 0.4pt}}}$}}}

\def\tr{\triangleright}

\newtheorem{theorem}{Theorem}
\newtheorem{definition}{Definition}

\newtheorem{proposition}[theorem]{Proposition}
\newtheorem{corollary}[theorem]{Corollary}
\newtheorem{example}{Example}
\newtheorem{remark}[example]{Remark}

\newenvironment{proof}[1][Proof]{\smallskip\noindent{\bf #1.}\quad}%
{\qed\par\medskip}

\date{}

\title{\Large \textbf{Rack shadows and their invariants}}

\author{Wesley Chang\footnote{Email: wesleychang1463@gmail.com} 
\and Sam Nelson\footnote{Email: knots@esotericka.org}} 

\begin{document}





\maketitle

\begin{abstract}
A \textit{rack shadow} or \textit{rack set} is a set $X$ with a rack 
action by a rack $R$, 
analogous to a vector space over a field. We use \textit{shadow colorings} 
of classical link diagrams to define enhanced rack counting invariants 
and show that the enhanced invariants are stronger than 
unenhanced counting invariants.
\end{abstract}

\begin{center}
\parbox{5.5in}{\textsc{Keywords:} quandles, racks, rack shadows,
link invariants, enhancements of counting invariants, shadow polynomials

\smallskip

}\end{center}

%

\section{\large \textbf{Introduction}}\label{I}

A \textit{rack} is a non-associative algebraic structure whose axioms 
correspond to the framed Reidemeister moves. \textit{Quandles} are specific 
types of racks whose axioms correspond to the three unframed 
Reidemeister moves. Every framed knot or link has a \textit{fundamental 
rack} described by generators and relations which may be read from any 
diagram of the knot or link. Similarly, every unframed knot or link 
has a \textit{fundamental quandle} \cite{J,FR,M}.
	 
The \textit{rack and quandle counting invariants} are integer-valued
invariants determined by the number of homomorphisms from the fundamental 
rack and quandle of a knot or link, which can be pictured as colorings 
of a knot or link diagram. Invariants of quandle and rack colored knot 
diagrams define \textit{enhancements} of the counting invariant; 
enhancements can also make use of extra structure of the coloring 
quandle or rack. In either case, enhancements specialize to the
original counting invariants and in most cases strengthen them. 
In this paper we will introduce enhancements of the rack counting 
invariants using sets with right actions by a rack we call \textit{rack 
shadows} (also known as \textit{rack sets}). 
	
The paper is organized as follows. In section \ref{B} we recall
the basics of racks, quandles, and their counting invariants. 
In section \ref{SB} we recall rack shadows and shadow colorings
of link diagrams. In section \ref{SBinv} we enhance the shadow counting
invariants with shadow polynomials and show that 
the enhanced invariants contain more information (and are thus stronger) 
than the unenhanced counting invariants. In section \ref{Q} we collect 
questions for future research. 

\smallskip

\section{\large \textbf{Racks and Quandles}}\label{B}

We begin with a definition from \cite{FR}.

\begin{definition} \textup{
A rack is a set R with two binary operations, $\rhd$ and $\rhd^{-1}$, 
that satisfy for  all $x,y,z \in R$:}
\begin{list}{}{}
\item[\textup{(i)}] $(x \rhd y) \rhd^{-1} y = x$ and 
$(x \rhd^{-1} y) \rhd y = x$, and
\item[\textup{(ii)}] $(x \rhd y) \rhd z = (x \rhd z) \rhd (y \rhd z)$. 
\end{list} 
\end{definition}

Racks can be understood as sets with binary operations in which right 
``multiplication'' by every element is an automorphism. It is a standard 
exercise to check that in any rack we also have
\begin{list}{}{}
\item[(i)] $(x \rhd^{-1} y) \rhd^{-1} z 
= (x \rhd^{-1} z) \rhd^{-1} (y \rhd^{-1} z)$,
\item[(ii)] $(x \rhd y) \rhd^{-1} z = (x \rhd^{-1} z) \rhd (y \rhd^{-1} z)$, and
\item[(iii)] $(x \rhd^{-1} y) \rhd z = (x \rhd z) \rhd^{-1} (y \rhd z)$.
\end{list}

\begin{definition}\textup{
A rack in which every $x\in R$ satisfies $x \rhd x = x$ is a \textit{quandle}.}
\end{definition}

\begin{example}
\textup{Standard examples of rack structures include:
\begin{list}{$\bullet$}{}
\item \textit{$(t,s)$-racks}: Modules over the ring 
$\tau=\mathbb{Z}[t^{\pm 1},s]/(s(t+s-1))$ with rack operation given by
\[x\tr y = tx+ sy.\]
If we set $s=1-t$, the result is a quandle known as an \textit{Alexander 
quandle}.
\item \textit{Coxeter racks}: Let $\mathbb{F}$ be a field, $V$ be an 
$\mathbb{F}$-vector space and
$(,):V\times V\to \mathbb{F}$ a symmetric bilinear form. Then the subset 
$C\subset V$ of nondegenerate vectors is a rack under
\[x\tr y = \alpha\left(x-2\frac{(x,y)}{(y,y)}y\right)\]
where $0\ne \alpha\in \mathbb{F}$.
If $\alpha=-1$ then $C$ is a quandle.
\item \textit{Constant action racks:} For any finite set $R=\{r_1,\dots,r_n\}$
and any permutation $\sigma\in S_n$ setting $r_i\tr r_j=r_{\sigma(i)}$
defines a rack operation. If $\sigma=\mathrm{Id}$ we have a quandle
called the \textit{trivial} quandle, $T_n$. $T_n$ is the only quandle 
in which the $\tr$ operation is associative.
\item \textit{The Fundamental Rack of a framed link} Let $L$ be an oriented 
framed link diagram. The fundamental rack of $L$ is the set of equivalence 
classes of rack words in a set of generators corresponding to arcs in $L$
under the equivalence relation determined by the rack axioms together with the
\textit{crossing relations} from the diagram.
\[\includegraphics{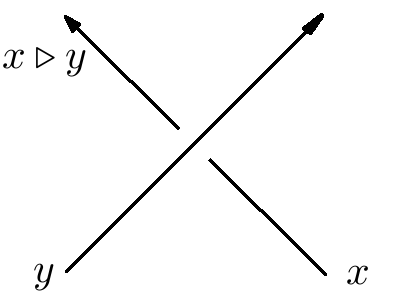} \quad
\includegraphics{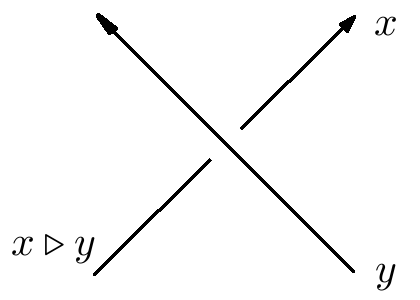}\]
As with other universal algebraic objects, we specify a fundamental rack
with a presentation  $\langle x_1,\dots, x_n\ |\ r_1,\dots,r_k\rangle$ 
where $x_i$ are generators and $r_i$ are relations, with the rack axiom 
relations understood.
\end{list}}
\end{example}

To every finite rack $R=\{r_1,\dots,r_n\}$ we associate an $n\times n$ matrix
$M_R$ encoding the operation table of the rack, called the \textit{rack matrix}.
Specifically, the entry in row $i$ column $j$ of $M_R$ is $k$ where 
$r_k=r_i\tr r_j$.

\begin{example}\label{ex1}\textup{
Let $R=\mathbb{Z}_4=\{1,2,3,4\}$ and set $t=1,$ $s=2$. Then $2(1+2-1)=4=0$
and we have a rack operation $x\tr y=x+2y$. The rack matrix is given by}
\[M_R=\left[\begin{array}{cccc}
3 & 1 & 3 & 1 \\
4 & 2 & 4 & 2 \\
1 & 3 & 1 & 3 \\
2 & 4 & 2 & 4 \\
\end{array}\right].\]
\end{example}

It follows easily from the rack axioms that every column in a rack matrix
must be a permutation of $\{1,2,\dots, n\}$. It is also true (see \cite{N3}, 
Corollary 2), 
though perhaps less obvious, that the diagonal of a rack matrix must be a 
permutation. The \textit{rack rank} of $R$ is the exponent of the diagonal
permutation regarded as an element of the symmetric group $S_n$.

\begin{example}
\textup{The rack in example \ref{ex1} has rack rank 2, since the diagonal
permutation is the transposition $(13)$.}
\end{example}

\begin{example}
\textup{Every quandle has diagonal permutation $\mathrm{Id}_{S_n}$ and thus 
rack rank 1.
Indeed, a finite rack is a quandle if and only if it has rack rank 1.}
\end{example}

\begin{definition}
\textup{Let $R$ and $R'$ be racks. A \textit{rack homomorphism} is a function
$f:R\to R'$ satisfying $f(x\tr y)=f(x)\tr f(y)$ for all $x,y\in R$. A bijective
rack homomorphism is a \textit{rack isomorphism}.}
\end{definition}

In the case that $R=FR(L)$ is the fundamental rack of an oriented framed 
link $L$, a unique homomorphism $f:R\to R'$ may be specified by assigning an 
element of $R'$ to each arc of $L$ such that the \textit{crossing condition}
\[\includegraphics{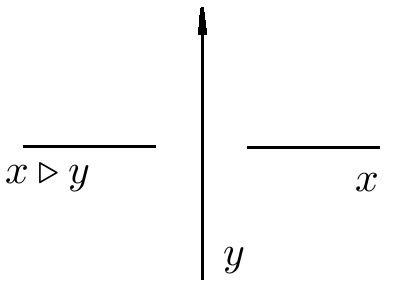}\]
is satisfied at every crossing. Such an assignment of elements of $R'$ to
arcs in $L$ is known as a \textit{rack coloring} of $L$ by $R'$. The
relationship between the rack axioms and the framed Reidemeister moves
guarantees that rack colorings are preserved by framed Reidemeister moves,
in the sense that a rack coloring of a diagram before a move corresponds
to a unique rack coloring of the diagram after the move.

Thus, if $T$ is a finite rack and $K$ is a framed oriented knot, 
the number of rack colorings $|\mathrm{Hom}(FR(K),T)|$ of $K$ by $T$ is
an invariant of framed isotopy. In \cite{N3} it was shown that for
finite racks $T$, the numbers of colorings of two framings of $K$ are
equal if the writhes of the diagrams differ by a multiple of the
rack rank $N$ of $T$. More generally, if $L=L_1\cup \dots\cup L_k$ is a link
with $k$ ordered components, then $w(L_i) \equiv w(L_i') \ \mathrm{mod}\ N$
for $i=1,\dots, k$ and $L$ ambient isotopic to $L'$ implies that the rack 
colorings
of $L$ and $L'$ by $T$ are in one-to-one correspondence. For a framing vector
$\mathbf{w}=(w(L_1),\dots,w(L_k))$ let us abbreviate $\prod_{i=1}^k q_i^{w_i}$ as
$q^{\mathbf{w}}$. Then we thus have:

\begin{definition}\label{def:rcinv}
\textup{Let $L=L_1\cup \dots \cup L_k$ be an oriented link with $k$ ordered 
components, $T$ a finite rack with rack rank $N$ and $W=(\mathbb{Z}_N)^k$. The \textit{integral 
rack counting invariant} is}
\[\mathrm{rc}(L,T)= \sum_{\mathbf{w}\in W} |\mathrm{Hom}(FR(L,\mathbf{w}),T)|\]
\textup{and the \textit{polynomial rack counting invariant} is}
\[\mathrm{prc}(K,T)=\sum_{\mathbf{w}\in W} |\mathrm{Hom}(FR(L,\mathbf{w}),T)|q^{\mathbf{w}}\]
\textup{where $(L,\mathbf{w})$ is a diagram of $L$ with writhe vector 
$\mathbf{w}=(w(L_1),\dots,w(L_k))$}.
\end{definition}

In  \cite{N3} it is shown that:
\begin{theorem}
Let $L$ and $L'$ be oriented links and $T$ a finite rack. If $L$ is
ambient isotopic to $L'$ then we have $rc(L,T)=rc(L',T)$ and
$prc(L,T)=prc(L',T)$.
\end{theorem}

Unlike groups, in which the trivial action is concentrated in a single 
identity element, trivial action is distributed throughout the structure 
in racks. We quantify this distribution with a \textit{rack polynomial}.

\begin{definition}
\textup{Let $R$ be a rack. The \textit{rack polynomial} of $R$ is the
two-variable polynomial}
\[\mathrm{rp}(R)=\sum_{x\in R} t^{c(x)}s^{r(x)}\]
\textup{where $c(x)=|\{y\in R\ |\ y\tr x = y\}|$ and
$r(x)=|\{y\in R\ |\ x\tr y = x\}|$.}
\end{definition}

We can easily compute the rack polynomial of a rack from its rack matrix
by counting the number of occurrences of the row number in each row and 
column. 
\begin{example}\textup{
The rack in example \ref{ex1} has rack polynomial $\mathrm{rp}(R)=2s^2+2t^4s^2$.
}\end{example}

In \cite{N} it is shown that:

\begin{theorem}
If $T$ and $T'$ are isomorphic racks, then $\mathrm{rp}(T)=\mathrm{rp}(T')$.
\end{theorem}

Given a subrack of a rack $R$, there is a \textit{subrack polynomial}
which captures information both about the subrack itself and how it is 
embedded in $R$.

\begin{definition}\textup{
Let $S\subset R$ be a subrack of $R$. The \textit{subrack polynomial}
of $S\subset R$ is given by}
\[\mathrm{srp}_{S\subset R}(s,t) = \sum_{x\in S}t^{c(x)}s^{r(x)}.\]
\textup{That is, we simply sum up the contributions to the full rack 
polynomial from the elements of the subrack.}
\end{definition}

\begin{example}
\textup{The subrack $S=\{1,3\}\subset R$  in example \ref{ex1} has
subrack polynomial $\mathrm{srp}_{S\subset R}(s,t)=2s^2$.}
\end{example}

Rack polynomials can be used to define an enhancement of the rack counting
invariant, as described in \cite{CN}. Specifically, for each rack coloring we 
find the image subrack and use its subrack polynomial as a ``signature''
of the coloring. The multiset of such polynomials paired with framing
information then yields an enhancement of the rack counting invariant; 
alternatively we can convert to a ``polynomial-style'' format.

\begin{definition}
\textup{Let $L=L_1\cup \dots\cup L_k$ be an oriented link with ordered
components, $W=(\mathbb{Z}_N)^k$ the space of framing vectors 
mod $N$, and $T$ a finite target rack with rack rank $N$. The 
\textit{multiset rack polynomial invariant} is the multiset}
\[\mathrm{mrp}(L,T)=\{ (\mathrm{srp}_{\mathrm{Im}(f)\subset T}(s,t)),\mathbf{w})\ 
|\ f\in \mathrm{Hom}(FR(L,\mathbf{w}),T),
 \ \mathbf{w}\in W\}\]
\textup{and the \textit{polynomial rack polynomial invariant} is}
\[\mathrm{prp}(L,T)=\sum_{\mathbf{w}\in W}
\left( \sum_{f\in \mathrm{Hom}(FR(L,\mathbf{w}),T)} 
z^{\mathrm{srp}_{\mathrm{Im}(f)\subset T}(s,t)}q^{\mathbf{w}}\right).\]
\end{definition}

\begin{example}
\textup{
The trefoil knot $3_1$ below has integral rack counting invariant 
$\mathrm{rc}(3_1)=6$, polynomial rack counting invariant 
$\mathrm{prc}(3_1)=4+2q$ and polynomial rack polynomial invariant 
$\mathrm{prp}(3_1)=2z^{2t^2} + 2z^{s^4t^2} + 2qz^{s^4t^2}$ with respect to the
rack $R$ from example \ref{ex1}.}
\[\begin{array}{ccc}
\includegraphics{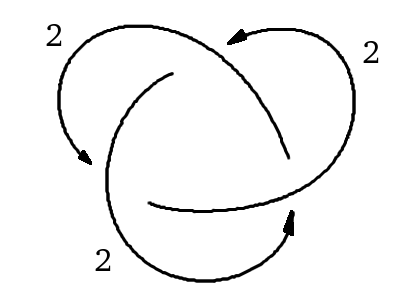} & \includegraphics{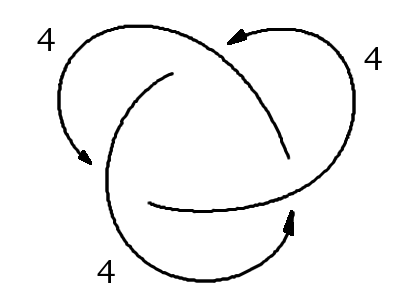} & \includegraphics{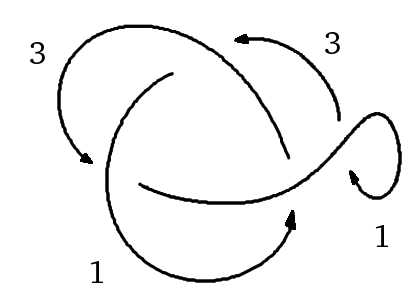} \\
\includegraphics{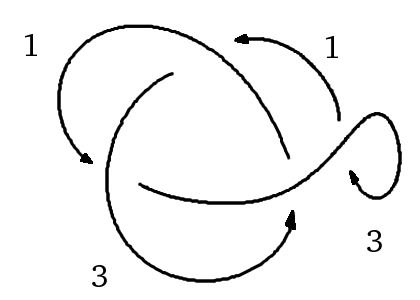} & \includegraphics{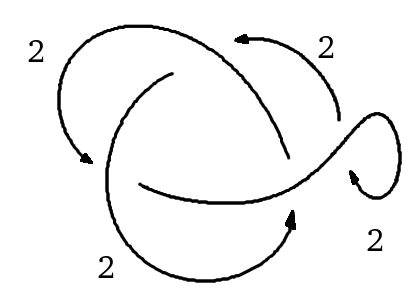} & \includegraphics{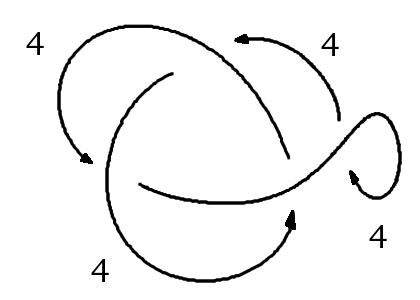} \\
\end{array}\]
\end{example}

\section{\large \textbf{Rack and Quandle Shadows}}\label{SB}

In this section  we define \textit{rack shadows} (also known as 
\textit{rack sets})  which can be used to generalize the  
\textit{shadow colorings} of knot diagrams by quandles described 
in previous work such as \cite{CKS, FRS, MN}.

\begin{definition}\textup{
Let $X$ be a set and $R$ a rack. A \textit{rack action} of $R$ on 
$X$ is an assignment of bijection $\cdot r:X \rightarrow X$ to each 
element of $R$ such that for all $x\in X$ and $r_1,r_2\in R$:}
\[(x\cdot r_1)\cdot r_2=(x\cdot r_2)\cdot (r_1 \rhd r_2).\]
\end{definition}

\begin{definition}\textup{ 
Let $R$ be a rack. A \textit{rack shadow} or \textit{$R$-shadow} is a 
set $X$ with a rack action by a rack $R$. The \textit{shadow matrix} of 
$X=\{x_1,\dots, x_m\}$ where $R=\{r_1,\dots, r_n\}$ is the $m\times n$ 
matrix whose $(i,j)$ entry is $k$ where $x_k=x_i\cdot r_j$.
A subset of $X$ closed under the action of $R$ is an $R$-\textit{subshadow}
of $X$.}
\end{definition}

\begin{example}
\textup{Let $R$ be any rack and let $X=\{x\}.$ Then $X$ is an $R$-shadow under
the shadow operation $x\cdot r=x$ for all $r\in R$, since we have}
\[(x\cdot r)\cdot r' = (x\cdot r)=x =(x\cdot r')=(x\cdot r')\cdot (r\tr r').\]
\textup{We will call this structure the \textit{singleton $R$-shadow}.}
\end{example}

\begin{example}
\textup{Let $R$ be any rack and let $X=R.$ Then $X$ is an $R$-shadow under
the shadow operation $x\cdot r=x\tr r$ for all $r\in R$, since we have}
\[(x\cdot r)\cdot r' = (x\tr r)\tr r'=(x\tr r')\tr (r\tr r')=
(x\cdot r')\cdot (r\tr r').\]
\textup{We will call this structure the \textit{rack $R$-shadow}. Note that
the shadow matrix of the rack $R$-shadow is the same as the rack matrix
of $R$.}
\end{example}

\begin{example}
\textup{Let $R$ be any rack and $X=\{x_1,\dots, x_n\}$ a set of cardinality 
$n$. Then for any permutation $\sigma\in S_n$ we have an $R$-shadow structure
given by $x_i\cdot r= x_{\sigma(i)}$:} 
\[(x\cdot r)\cdot r' = x_{\sigma(i)} \cdot r' = x_{\sigma^2(i)}
=x_{\sigma(i)} \cdot (r\tr r')=(x_{i}\cdot r') \cdot (r\tr r').\]
\textup{We will call this shadow structure a \textit{constant action shadow}.}
\end{example}

\begin{example}
\textup{For a less trivial example of an $R$-shadow structure, let
$R$ be the $(t,s)$-rack from example \ref{ex1}, i.e. $R=\mathbb{Z}_4$
with $t=1$, $s=2$ and $x\tr y=tx+sy=x+2y$. Then the three element set
$X=\{1,2,3\}$ is an $R$-shadow with operation matrix}
\[\left[\begin{array}{cccc}
1 & 2 & 1 & 2 \\
2 & 3 & 2 & 3 \\
3 & 1 & 3 & 1
\end{array}\right].
\]
\end{example}

\begin{definition}\textup{
Let $R$ be a rack. Two $R$-shadows $X, Y$ are \textit{isomorphic} if 
there is a bijection $\phi : X \rightarrow Y$ such that 
$\phi(x\cdot r)= \phi (x) \cdot r$ for all $r \in R.$}
\end{definition}

\begin{proposition}
Let $R$ be a rack and let $X$ and $Y$ 
be $n$-element constant action $R$-shadows with permutations 
$\sigma,\tau\in S_n$ respectively. Then $X$ is isomorphic 
to $Y$ if and only if $\sigma$ is conjugate to $\tau$ in $S_n$.
\end{proposition}

\begin{proof}
Suppose $\phi:X\to Y$ is an $R$-shadow isomorphism. Then $\phi(x_i)=y_j$
induces a permutation $\Phi\in S_n$ defined by $\Phi(i)=j$, so that
$\phi(x_i)=y_{\Phi(i)}$ and we have
\[ \phi(x_i\cdot r)=\phi(x_{\sigma(i)})=y_{\Phi(\sigma(i))}
\quad \mathrm{and} \quad
\phi(x_i)\cdot r=y_{\Phi(i)}\cdot r=y_{\tau(\Phi(i))}\]
and hence $\Phi\circ \sigma \circ \Phi^{-1}=\tau$. 

Conversely, $\tau$ conjugate to $\sigma$ in $S_n$ by $\Phi$ in $S_n$ implies
that $\Phi$ induces an isomorphism $\phi:X\to Y$ by $\phi(x_i)=y_{\Phi(i)}$.
\end{proof}

\begin{definition}
\textup{Let $L$ be a framed oriented link diagram and $X$ an $R$-shadow. 
A \textit{shadow coloring} of $L$ by $R$ is an assignment of elements
of $R$ to arcs in $L$ and elements of $X$ to the regions between the arcs
such that at every crossing and along every arc we have}
\[\includegraphics{wc-sn-4.png}\quad \includegraphics{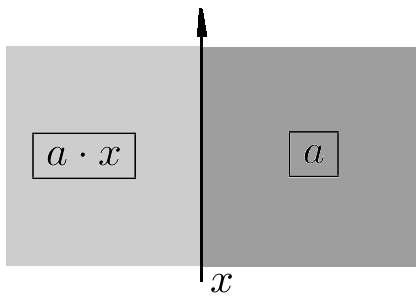}\]
\textup{Shadow colors will be indicated by \framebox{boxes}. Note that the 
requirement that $R$ acts on $X$ via a rack action, i.e.
the rack shadow axiom, is precisely the condition needed to guarantee that 
shadow colorings are well-defined at crossings:}
\[\includegraphics{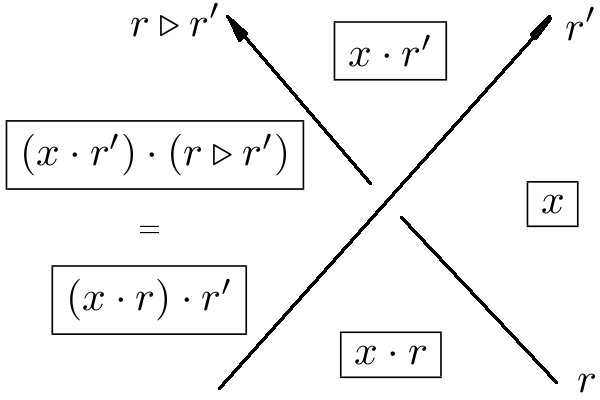}\quad \quad \quad 
\includegraphics{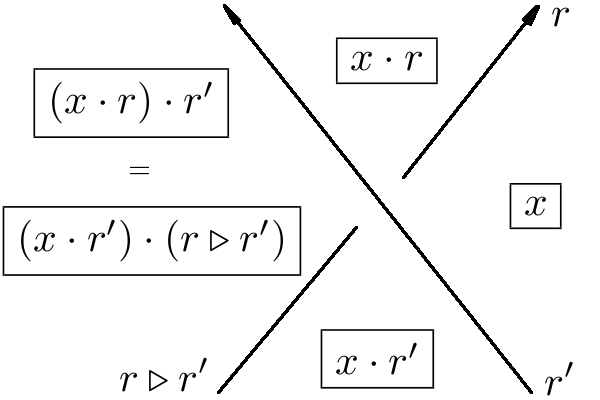}\]
\end{definition}

\begin{proposition}
Let $L$ be a link diagram, $R$ a rack and $X$ an $R$-shadow. Then for each 
rack coloring of $L$ by $R$ and each element of $X$ there is exactly one shadow 
coloring of $L$.
\end{proposition}

\begin{proof}
This is well-known in the case that $X=R$ is the rack $R$-shadow; see
\cite{MN} for instance. For the case of a general $R$-shadow $X$, consider 
a rack coloring of $L$ and choose a region 
of $L$. Any element of $X$ can be assigned to the chosen region, and any
such choice determines a unique shadow coloring by pushing the color across
the arcs with $\cdot$ or $\cdot^{-1}$.
\end{proof}

\begin{definition}
\textup{Let $L$ be a link diagram, $R$ a rack and $X$ an $R$-shadow. The
\textit{shadow counting invariant} $\mathrm{sc}(L)$ is the number of shadow 
colorings of $L$ by $X$.}
\end{definition}

\begin{corollary}\label{cor1}
The shadow counting invariant of a link $L$ by an $R$-shadow $X$ is
given by \[\mathrm{sc}(L)=|X|\mathrm{rc}(L,R).\]
\end{corollary}

Just as racks have rack polynomials, we can define a polynomial invariant of
$R$-shadows:
 
\begin{definition} \textup{
The \textit{$R$-shadow polynomial} $\mathrm{rsp}(X)$ of 
an $R$-shadow $X$ is the sum}
\[\mathrm{rsp}(X)=\sum_{x \in X} t^{r(x)}\] 
\textup{where $r(x) = |\{r \in R \ | \ x \cdot r = x\}|$.
If $S\subset X$ is a subshadow then the \textit{subshadow polynomial} of 
$s$ is}
\[\mathrm{ssp}_{S\subset X}(t)=\sum_{x \in S} t^{r(x)}\] 
\end{definition}

\begin{proposition}
If two $R$-shadows $X$ and $Y$ are isomorphic, we have 
$\mathrm{rsp}(X)=\mathrm{rsp}(Y).$
\end{proposition}

\begin{proof}

Suppose $\phi:X\to Y$ is an $R$-shadow isomorphism. Then $r(\phi(x))=r(x)$
and the contribution to $\mathrm{rsp}(X)$ from $x\in X$ is equal to the 
contribution to $\mathrm{rsp}(Y)$ from $y=\phi(x)\in Y$. Since every $y\in Y$
is equal to $\phi(x)$ for some $x\in X$ and conversely, it follows that 
the sums are equal.
\end{proof}

\section{\large \textbf{Shadow Enhanced Counting Invariants}}\label{SBinv}

Unfortunately, corollary \ref{cor1} implies that the unenhanced shadow
counting invariant does not contain any more information than the
ordinary rack counting invariant. If we want to exploit rack shadows,
then, we must look to enhancements.

For the rack $R$-shadow $X=R$ in the case that $R$ is quandle, this is 
done in \cite{CKS,MN} with cocycles in the third cohomology of the 
target quandle $H^3_Q(R)$. These invariants have a natural interpretation
in terms of quandle colorings of knotted surfaces in $\mathbb{R}^4$.
In this section we will take a different 
approach, using $R$-shadow polynomials to define an enhancement.

\begin{definition}
\textup{For a shadow coloring $f$ of a link diagram $L$ by an $R$-shadow $X$,
the closure of the set of shadow colors under the action of the image
subrack $\mathrm{Im}(f)\subset R$ of $f$ is a subshadow called the 
\textit{shadow image} of $f$, denoted $\mathrm{si}(f)$.}
\end{definition}

It is easy to see that the shadow image of a shadow colored link diagram
is invariant under framed Reidemeister moves, and thus its subshadow 
polynomial can be used as a signature of the shadow coloring. We then have:

\begin{definition}
\textup{Let $L=L_1\cup L_2\cup \dots \cup L_k$ be an oriented link with
$k$ components, $R$ a finite rack with rack rank $N$, $W=(\mathbb{Z}_N)^k$ 
and $X$ an $R$-shadow. The \textit{multiset shadow polynomial invariant} 
of $L$ with respect to the $R$-shadow $X$ is the multiset}
\[\Phi(L)=\{(\mathrm{ssp}_{\mathrm{si}(f)\subset X}(t),\mathbf{w})\ | \ \mathbf{w}\in 
(\mathbb{Z}_N)^k,\ f \ \mathrm{shadow}\ \mathrm{coloring}\}\]
\textup{and the \textit{$R$-shadow polynomial invariant} of $L$ with 
respect to $X$ is}
\[\mathrm{sp}(L)=\sum_{\mathbf{w}\in W} 
\left(\sum_{f \mathrm{ \ shadow \ coloring}} 
z^{\mathrm{ssp}_{\mathrm{si}(f)\subset X}(t)}q^{\mathbf{w}}\right). \]
\end{definition}

\begin{example}
\textup{If $X$ is the singleton shadow $X=\{x\}$ then 
$\mathrm{sp}(L)=z^t \mathrm{rp}(L)$ since shadow colorings in this case 
are simply rack colorings with every region colored \framebox{$x$} and the
subshadow polynomial of every shadow image is equal to $t$.}
\end{example}

\begin{remark}
\textup{The shadow polynomial specializes to the shadow counting invariant
$\mathrm{sc}(L,T)$ by setting $t=0$ (or equivalently $z=1$). It follows that 
the shadow polynomial is at least as strong an invariant as the integral rack 
counting invariant. The following example demonstrates that the shadow 
polynomial is stronger than the unenhanced shadow counting invariant.}
\end{remark}

\begin{example}
\textup{The two knots below, $5_1$ and $6_1$, both have shadow counting 
invariant $\mathrm{sc}(5_1,R)=60=\mathrm{sc}(6_1,R)$ with respect to the 
$R$-shadow $X=\{1,2\}$ where}
\[
M_X=\left[\begin{array}{cccccccccc}
1 & 1 & 1 & 1 & 1 & 2 & 2 & 2 & 2 & 2 \\
2 & 2 & 2 & 2 & 2 & 1 & 1 & 1 & 1 & 1
\end{array}\right],\ M_R=
\left[\begin{array}{cccccccccc} 
1 & 3 & 5 & 2 & 4 & 3 & 1 & 4 & 2 & 5 \\
5 & 2 & 4 & 1 & 3 & 5 & 3 & 1 & 4 & 2 \\
4 & 1 & 3 & 5 & 2 & 2 & 5 & 3 & 1 & 4 \\
3 & 5 & 2 & 4 & 1 & 4 & 2 & 5 & 3 & 1 \\
2 & 4 & 1 & 3 & 5 & 1 & 4 & 2 & 5 & 3 \\
8 & 9 & 10 & 6 & 7 & 6 & 10 & 9 & 8 & 7 \\
7 & 8 & 9 & 10 & 6 & 8 & 7 & 6 & 10 & 9 \\
6 & 7 & 8 & 9 & 10 & 10 & 9 & 8 & 7 & 6 \\
10 & 6 & 7 & 8 & 9 & 7 & 6 & 10 & 9 & 8 \\
9 & 10 & 6 & 7 & 8 & 9 & 8 & 7 & 6 & 10 
\end{array}\right].
\]
\textup{However, the shadow enhanced invariant distinguishes the knots:}
\[
\begin{array}{cc}
\includegraphics{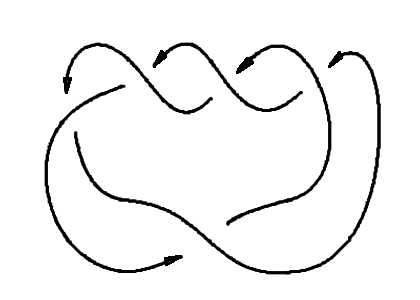} & \includegraphics{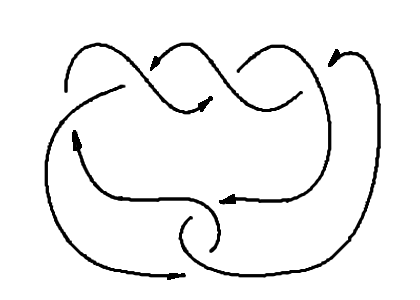} \\
\mathrm{sp}(5_1)=10z + 10z^t + 40z^{t^5} &
\mathrm{sp}(6_1)=50z + 10z^t. \\ 
\end{array}
\]
\end{example}

\section{\large \textbf{Questions for Further Research}}\label{Q}

In this section we collect a few questions for further research. Much 
remains to be done in the study of rack shadows and their invariants.

It is natural to think of a rack action as akin to scalar 
multiplication. What kinds of futher enhancements result when
an $R$-shadow $X$ has extra structure of its own, e.g. when $X$ is
a group or another quandle?  Work is already underway on
quandle/rack homology with coefficients in an $R$-shadow $X$, 
generalizing the 3-cocycle shadow coloring invariants found in 
\cite{CKS, MN}, for instance.

What is the relationship between the shadow polynomial enhanced invariant
determined by an $R$-shadow $X$ and the rack polynomial enhanced invariant
determined by $R$ in \cite{CN}? In the special case where $X=R$, the 
shadow image polynomial is just the subrack polynomial of $\mathrm{Im}(f)$
with $s=1$.

More generally, every rack can be decomposed as a disjoint union of orbit 
subracks which act on each other via rack actions (see \cite{NW}); what is 
the relationship between the invariants defined by the orbit racks,
the rack shadows, and the overall rack?

Many racks and quandles have additional structure such as abelian groups,
modules over various rings, etc. What are some examples of $R$-shadow 
structures defined on abelian groups or modules?

\texttt{python} code for computing the invariants described in this paper
is available from the second author's website at 
\texttt{http://www.esotericka.org}.

\noindent
\textsc{Department of Mathematical Sciences, \\
Claremont McKenna College, \\
850 Columbia Ave., \\
Claremont, CA 91711}

\end{document}